\documentclass[12pt]{article}
\usepackage{enumitem}
\usepackage{soul} 
          
\usepackage{color}
\usepackage{graphicx,subfigure,amsmath,amssymb,amsfonts,bm,epsfig,epsf,url,dsfont}
\usepackage{times}
\usepackage{bbm}
\usepackage{booktabs}
\usepackage{cases}
\usepackage{hyperref}

\renewcommand{\phi}{\varphi}

\newcommand{\E}{\mathbb{E}}
\newcommand{\N}{\mathbb{N}}
\newcommand{\R}{\mathbb{R}}

\newcommand{\cK}{\mathcal{K}}

\newcommand{\cN}{\mathcal{N}}

\def\ds1{\mathds{1}}
\renewcommand{\epsilon}{\varepsilon}
\newcommand{\eps}{\epsilon}

\newcommand{\argmin}{\mathop{\mathrm{argmin}}}

\renewcommand{\tilde}{\widetilde}

\newlength{\minipagewidth}
\newcommand{\beq}{\begin{equation}}
\newcommand{\eeq}{\end{equation}}

\newcommand{\beqa}{\begin{eqnarray}}
\newcommand{\eeqa}{\end{eqnarray}}

\newcommand{\beqan}{\begin{eqnarray*}}
\newcommand{\eeqan}{\end{eqnarray*}}

\def\ba#1\ea{\begin{align*}#1\end{align*}} 
\def\banum#1\eanum{\begin{align}#1\end{align}} 

\def \nnz {\mathrm{nnz}}

\def\eps{\varepsilon}
\newtheorem{theorem}{Theorem}

\newtheorem{lemma}{Lemma}

\newtheorem{remark}{Remark}

\newtheorem{definition}{Definition}

\usepackage{natbib}
 \usepackage[top=1in,bottom=1in,left=1in,right=1in]{geometry}
 \usepackage[small,bf]{caption}
 \newcommand{\BlackBox}{\rule{1.5ex}{1.5ex}}  
 \newenvironment{proof}{\par\noindent{\bf Proof\ }}{\hfill\BlackBox\\[2mm]}

\global\long\def\norm#1{\left\Vert #1\right\Vert }

\newcommand{\red}[1]{#1}

\newif\ifshortversion

\newcommand{\mideq}[1]{$$#1$$}

\newcommand{\fullversion}[1]{#1}

\begin{document}

\title{An homotopy method for $\ell_p$ regression provably beyond self-concordance and in input-sparsity time}

\author{S\'ebastien Bubeck
\and Michael B. Cohen
\and 
Yin Tat Lee 
\and 
Yuanzhi Li}
\maketitle
\begin{abstract}
We consider the problem of linear regression where the $\ell_2^n$ norm loss (i.e., the usual least squares loss) is replaced by the $\ell_p^n$ norm. We show how to solve such problems in $O_p(n^{|1/2 - 1/p|} \log^{O(1)}(1/\eps))$ (dense) matrix-vector products and $O_p(\log^{O(1)}(1/\eps))$ matrix inversions, or in $O_p(n^{|1/2 - 1/p|} \log^{O(1)}(1/\eps))$ calls to a (sparse) linear system solver. This improves the state of the art for any $p\not\in \{1,2,+\infty\}$. Furthermore we also propose a randomized algorithm solving such problems in {\em input sparsity time}, i.e., $O_p((Z + \mathrm{poly}(d))\log^{O(1)}(1/\eps))$ where $Z$ is the size of the input and $d$ is the number of variables. Such a result was only known for $p=2$. Finally we prove that these results lie outside the scope of the Nesterov-Nemirovski's theory of interior point methods by showing that any symmetric self-concordant barrier on the $\ell_p^n$ unit ball has self-concordance parameter $\tilde{\Omega}(n)$.
\end{abstract}

%
%

\section{Introduction}
Linear programming is concerned with optimization problems of the form:
\mideq{\min_{x \in \R^d_+ : \ A x = b} c \cdot x ~,}
for some $A \in \R^{n \times d}, b \in \R^n, c \in \R^d$. Such problem often comes with a guarantee on how far a solution can be, in which case the problem can be rewritten as (up to rescaling)
$$\min_{x \in \R^d : \ \|x\|_{\infty} \leq 1 \ \text{and} \ A x = b} c \cdot x ~.$$
Classical interior point methods show that such problems can be solved up to machine precision via solving $\sqrt{d}$ linear systems, see e.g., \cite{NN94}. 

In this paper we investigate the complexity of replacing the $\ell_{\infty}$ constraint by an $\ell_p$ constraint, $1<p<+\infty. $\fullversion{That is:
\begin{equation} \label{eq:classicalformulation} 
\min_{x \in \R^d : \ \|x\|_{p} \leq 1 \ \text{and} \ A x = b} c \cdot x ~.
\end{equation}
}The case of $\ell_2$ exactly corresponds to solving a linear system. Moreover for the Euclidean ball $\{x : \|x\|_2 \leq 1\}$ there exists a barrier with self-concordance $\nu=1$, and thus the Nesterov-Nemirovski's interior point methods theory correctly predicts that the case $p=2$ can be solved in a dimension-free number of iterations. Our contribution is to show that the Nesterov-Nemirovski theory is provably suboptimal for any $p \not\in \{1,2,\infty\}$. More precisely we show that for any $p \neq 2$, any {\em symmetric} self-concordant barrier on $\{x \in \R^d : \|x\|_p \leq 1\}$ has self-concordance parameter at least roughly $d$. On the other hand we propose a new homotopy method which requires only $O^*(n^{|1/2 - 1/p|})$ iterations\footnote{We use the notation $O^*$ to hide polynomial factors in $p^2/(p-1)$ and polylogarithmic terms.}
\fullversion{
\footnote{We note that the dual problem to \eqref{eq:classicalformulation} corresponds to the $\ell_q$ norm problem (where $1/p+1/q=1$) but we shall not use this fact and we treat any $p 
\in (1,+\infty)$ (observe that $|1/2-1/p|=|1/2-1/q|$).}}
, thus interpolating between the known results for $p \in \{1,2,+\infty\}$. 

Curiously, our homotopy method runs in $O^*(n^{|1/2 - 1/p|})$ calls to a sparse linear system solver (if $A$ is sparse), or alternatively in $O^*(n^{|1/2 - 1/p|})$ dense matrix-vector products and $O^*(1)$ matrix inversions. \fullversion{Although our result does not imply such result for $p=\infty$, we note that there is no known algorithm for $p=\infty$ (i.e., for linear programming) with the latter running time, and the best result in this direction involves $O^*(1)$ matrix inversions of $d \times d$ size and many matrix inversions of smaller matrices \cite{lee2015efficient}.
}
On top of the results above, we also show how to combine this new method with recent advances in accelerated stochastic gradient descent to obtain an algorithm running in input sparsity time, namely a running time of the form $O^*(Z + d^c)$ where $c$ depends on $p$ and $Z$ the number of non-zeros in $A$. Unfortunately our approach does not {\em a priori} shed light on an input sparsity time algorithm for the case $p=\infty$ since our running time explodes as $p\rightarrow \infty$. Such a result would in our opinion be a major breakthrough.

In the rest of the paper we consider the following equivalent problem, which we call $\ell_p$ regression:
\begin{equation} \label{eq:regressionformulation}
\min_{x \in \R^d} c \cdot x + \|A x - b\|_p^p ~.
\end{equation}
Observe that to have a bounded solution we need to assume that $c \in \mathrm{ker}(A)^{\perp}$. 
\fullversion{
Note also that \eqref{eq:classicalformulation} can essentially be reduced to \eqref{eq:regressionformulation} by using the matrix $\left( {\begin{array}{cc}
   \lambda A & 0 \\
   0 & \mu I_d \\
  \end{array} } \right)$ and target vector $\left( {\begin{array}{cc}
   \lambda b \\
   0 \\
  \end{array} } \right)$ for some $\lambda>0$ large enough and some $\mu>0$. In particular note that the parameter regime for $(n,d)$ is inversed compared to the discussion above, namely in \eqref{eq:classicalformulation} one had $d \geq n$ whereas for the rest of the paper we have $n \geq d$ (and in fact $n$ potentially much larger than $d$).

\red{Recently, there are a lot of progress for the $\ell_p$ regression for the case of $n\gg d$ \cite{cohen2015p,woodruff2013subspace,meng2013low,clarkson2013low,clarkson2016fast,sohler2011subspace,dasgupta2009sampling}. These results show various ways to find a matrix $A'$ with fewer rows such that $\|Ax\|_{p}\approx\|A'x\|_{p}$ for all vectors $x \in \mathbb{R}^d$. In particular, \cite{cohen2015p} shows that one can find such $A'$ with only roughly $d^{\max(p/2,1)}$ many rows by sampling rows of $A$ and rescaling. As a result, they show how to solve $\ell_p$ regression with $1+\eps$ multiplicative error in time $\tilde{O}(Z + d^{\max(p/2,1)+1}/\eps^{O(1)} + d^3)$ time.
For the case $p>2$, our runtime in Theorem \ref{thm:input_sparsity} is better in both dimension dependence and $\eps$ dependence. However, the $\log(1/\eps)$ dependence comes with a cost that our runtime is invariant under the conjugate transform $p \rightarrow \frac{p}{p-1}$. Therefore, our algorithm in the case $p < 2$ is much worse than existing results.
}

In Section \ref{sec:hom} we give our new homotopy method to solve \eqref{eq:regressionformulation}. In Section \ref{sec:inputsparsity}, we give our input sparsity time algorithm. Finally in Section \ref{sec:lb} we prove the $\tilde{\Omega}(n)$ lower bound on the self-concordance parameter for symmetric barriers on the $\ell_p^n$ ball. 
}
%

\section{An homotopy method for $\ell_p$ regression} \label{sec:hom}
\red{The main difficulty in optimizing the $\ell_p$ norm is its behavior around $0$, namely its second derivative either does not exist (for $p<2$) or equals to $0$ (for $p>2$).} We resolve this issue by gradually modifying the $\ell_p$ norm around $0$ (starting with a large modification, and slowly reducing it). This idea falls in the general framework of {\em homotopy methods}. 

\subsection{Smoothing family and homotopy path}
To develop our homotopy method we introduce a family $(f_t)_{t \geq 0}$ of functions approximating $s \mapsto |s|^p$ with the following properties: (i) $f_0(s) = |s|^p$, (ii) $s \mapsto f_t(s)$ is quadratic on $[-t,t]$, and (iii) $s \mapsto f_t(s)$ and $t \mapsto f_t(s)$ are $C^1$. We realize this with the following family:
\[
f_{t}(s)=\begin{cases}
\frac{p}{2}t^{p-2}s^{2} & \text{if }|s|\leq t ,\\
|s|^{p}+(\frac{p}{2}-1)t^{p} & \text{otherwise.}
\end{cases}
\]
\red{We construct this function by replacing the function $|s|^p$ by a quadratic function on $\{s: |s|\leq t\}$ and shifting the function outside to make sure it is twice differentiable. We note  that our framework works for many other families of functions and we choose this mainly for its simple formula.}

We also use a slight abuse of notation and write for a vector $s = (s_1, \hdots, s_n)$, \mideq{f_t(s) :=(f_t(s_1), \hdots, f_t(s_n)).} Next we define the homotopy path as follows:
\begin{equation}
x(t) := \argmin_{x \in \R^d, \ x \in \mathrm{ker}(A)^{\perp}} c \cdot x+\sum_{i=1}^{n}f_{t}(s_{i}(x))\label{eq:xt_def}
\end{equation}
where $s(x)=Ax-b$ (we also use the notation $s(t) = Ax(t) - b$).

The key observation is that the path $(x(t))_{t>0}$ is ``easy to follow", namely, $f_{(1-h)t}$ remains well conditioned on a neighborhood of $x(t)$ which contains $x((1-h) t)$ for some constant $h$ (which depends only on $p$). We introduce the following notion of neighborhood, for $s \in \R^n$ and $\gamma \in \R$,
$$\cN_s(\gamma) := \{s' \in \R^n : \forall i \in [n], \left| (s_i')^{p/2} - (s_i)^{p/2} \right| \leq \gamma\} ~.$$

\begin{lemma} \label{lem:cond}
For any $0 \leq h \leq \frac1{2p}$, we have $s((1-h) t) \in \cN_{s(t)}(\gamma)$
where 
$$\gamma = \left(1+\frac{p^3}{p-1} \sqrt{n} h\right) t^{p/2}.$$
Furthermore, for all $s \in \cN_{s(t)}(\gamma)$, one has
\mideq{D_t \preceq \nabla^2 f_{(1-h)t}(s) \preceq \kappa D_t ~,}
where $D_t$ is the diagonal matrix whose $i^{th}$ diagonal entry is
$$\frac{p-1}{2} \max(t^{p/2}, |s_i(t)|^{p/2} - \mathrm{sign}(p-2)\gamma)^{2 - 4/p}$$ 
and $\kappa = \frac{2p^2}{p-1}\left(3 + \frac{2p^3}{p-1} \sqrt{n} h \right)^{|2 - 4 /p|}$.
\end{lemma}

Using the notation $O_p$ to hide polynomial factors in $p^2/(p-1)$, we have $\kappa = O_p(n^{|1-2/p|})$ for $0 \leq h \leq \frac{1}{2p}$. \red{The ratio between the upper bound and the lower bound of the Hessian is called the condition number. This number is important due to the following theorem:}
\begin{lemma}[\cite{Nes04}]
\label{lem:AGD} Given a convex function $f$ satisfies $D \preceq \nabla^2 f(x) \preceq \kappa D$ for all $x \in \mathbb{R}^n$ with some given fixed diagonal matrix $D$ and some fixed $\kappa$. Given an initial
point $x_{0}$ and an error parameter $0<\eps<\frac{1}{2}$, the accelerated gradient descent (AGD) outputs
$x$ such that
\[
f(x)-\min_{x}f(x)\leq\eps(f(x_{0})-\min_{x}f(x))
\]
in $O(\sqrt{\kappa} \log(\kappa / \eps))$ iterations. Each iteration involves computing $\nabla f$ at some point $x$ and some linear-time calculations.
\end{lemma}

\red{Therefore, if the condition number of $f_{(1-h)t}$ was valid globally (instead of merely on the neighborhood $\cN_{s(t)}(\gamma)$) then we would apply AGD to find $x(t)$ in $O_p(n^{|1/2-1/p|} \log(n/\eps))$ iterations. 
}

The proof of the above lemma is really the key to our homotopy method, however it is rather tedious calculation and thus we postpone it to Section \ref{sec:prooflemcond}. \red{We remark that the choice of neighborhood is forced by how much $x(t)$ can change in the worst case when we change $t$ and $D_t$ is chosen such that it is close to $\nabla^2 f_t(s)$. Therefore, despite this being the key lemma, its formulation is automatic given the choice of $f_t$. We suspect the choice of $f_t$ does not matter too much either given that it needs to be close to $x^p$.}

Fortunately there is a rather simple idea to actually make the condition number valid globally, namely to {\em extend smoothly} the function outside of the good neighborhood.

\subsection{Algorithm}
To describe our algorithm, we introduce the following definitions to extend $f_t$ smoothly outside the range $[\ell,u]$.
\begin{definition} \label{def:1}
For any positive $t$ and $0\leq \ell \leq u$, we define $f_{t,\ell,u}$ to be the ``quadratic extension" of $f_t$ on $[\ell,u]$, more precisely:
\[
f_{t,\ell,u}(s) := \begin{cases}
f_{t}(s) & \text{if }\ell\leq s\leq u\\
f_{t}(u)+f_{t}'(u)(s-u)+\frac{1}{2}f_{t}''(u)(s-u)^{2} & \text{if } s\geq u\\
f_{t}(\ell)+f_{t}'(\ell)(s-\ell)+\frac{1}{2}f_{t}''(\ell)(s-\ell)^{2} & \text{otherwise}
\end{cases}.
\]
\end{definition}

Note that both the (global) smoothness and strong convexity of $f_{t,\ell,u}$ is equal to the one of $f_{t}$ restricted to $[\ell,u]$. Furthermore by strict convexity of $f_t$ and $f_{t,\ell,u}$ one has that for any convex function $\phi$, the functions $\R^n \ni s \mapsto \phi(s)+\sum_{i=1}^n f_t(s_i)$ and $s \mapsto \phi(s)+\sum_{i=1}^n f_{t,\ell_i,u_i}(s_i)$ admit the same unique minimizer $s^*$ provided that $\forall i \in [n], s^*_i \in [\ell_i, u_i]$.

\red{Although the Hessian of $f_{t,\ell,u}$ in the $s$ variables is well-conditioned, it might be ill-conditioned in the $x$ variables (namely, its Hessian is not close to a diagonal matrix globally). To apply AGD (Lemma \ref{lem:AGD})}, we need to do a change of variables according to $A$ as follows:

\begin{definition} \label{def:2}
Let $D_t$ be the diagonal matrix defined in Lemma \ref{lem:cond} and $P_t$ be the preconditioner defined by
$((A^{\top} D_t A)^{\dagger})^{1/2}$. 

We introduce the functions $\tilde{f}_t$ and $g_t$, respectively the quadratic extension of $f_{(1-h) t}$ on a well-chosen hyperrectangle and its preconditioned version (with the cost vector $c$):
$$\tilde{f}_{t}(s) := \sum_{i=1}^n \tilde{f}_{t,i}(s)\quad\text{ with }\quad \tilde{f}_{t,i}(s)  := f_{(1-h) t, (|s_i(t)|^{p/2} - \gamma)^{2/p}, (|s_i(t)|^{p/2} + \gamma)^{2/p}}(s_i) ~,$$
and
\mideq{g_t(y) := c \cdot P_t y + \tilde{f}_{t}(s(P_t y)) ~.}
\end{definition}
\begin{remark}
\red{The hyperrectangle we used to define the quadratic extension is exactly $\cN_{s(t)}(\gamma)$. Therefore, we have that $\tilde{f}_{t} = f_{(1-h) t}$ on $\cN_{s(t)}(\gamma)$. Although $\tilde{f}_{t}$ depends on $\gamma$, we choose not to indicate it in the symbol for notation simplicity.}
\end{remark}
Using Lemma \ref{lem:cond} and simple calculations, we obtain:
\begin{lemma} \label{lem:condg} 
With the notations of Definition \ref{def:2}, we have for all $y \in \mathbb{R}^d$,
\[
Q_{t}\preceq\nabla^{2}g_{t}(y)\preceq\kappa\cdot Q_{t}
\]
where $Q_{t}$ is the orthogonal projection matrix $(A^{\top}D_{t}A)^{\frac{1}{2}}(A^{\top}D_{t}A)^{\dagger}(A^{\top}D_{t}A)^{\frac{1}{2}}$.
\end{lemma}
\ifshortversion \else
\begin{proof}
Note that 
\[
\nabla^{2}g_{t}(y)=P_{t}^{\top}A^{\top}\Sigma(y)AP_{t}
\]
where $\Sigma(y)$ is the diagonal matrix whose diagonal is $\nabla^{2}\tilde{f}_{(1-h)t}(s(P_{t}y))$.
By the construction of $\tilde{f}$, the global smoothness and strong
convexity of $\tilde{f}_{(1-h)t}$ is same as $f_{(1-h)t}$ restricted
to $\cN_{s(t)}(\gamma)$. Hence, Lemma \ref{lem:cond} shows that
\mideq{D_{t}\preceq\nabla^{2}\tilde{f}_{(1-h)t}(s(P_{t}y))\preceq\kappa D_{t}.}
Therefore, we have that
\[
P_{t}^{\top}A^{\top}D_{t}AP_{t}\preceq\nabla^{2}g_{t}(y)\preceq\kappa\cdot P_{t}^{\top}A^{\top}D_{t}AP_{t}.
\]
The result follows from the equation $P_{t}^{\top}A^{\top}D_{t}AP_{t}=(A^{\top}D_{t}A)^{\frac{1}{2}}(A^{\top}D_{t}A)^{\dagger}(A^{\top}D_{t}A)^{\frac{1}{2}}$.
\end{proof}
\fi

Our algorithm can now be described as follows, where $t_0>0$ and $h = \frac1{2p}$ are parameters, and $t_k := (1-h)^k t_0$ for $k \in \N$.
\begin{itemize}
\item Find $x(t_0)$ using Lemma \ref{lem:init} and compute $P_{t_0}$.
\item For $k = 0, 1, \cdots, O^*(1)$
\begin{itemize}
\item Given $x(t_k)$ and $P_{t_{k}}$, run accelerated gradient descent (AGD) on $g_{t_k}$ to obtain an approximate of $y(t_{k+1})=P_{t_k}^{-1} x(t_{k+1})$.
\item Compute $x(t_{k+1})$ by the formula $x(t_{k+1}) = P_{t_k} y(t_{k+1})$, and compute $P_{t_{k+1}}$.
\end{itemize}
\end{itemize}
Each phase of the algorithm requires a matrix inversion to compute $P_t$, and each iteration in a run of AGD requires a (dense) matrix-vector product. Theorem \ref{th:main} below shows that in total the algorithm needs $O^*(1)$ matrix inversion and $O^*(n^{|1/2-1/p|})$ dense matrix-vector products.

A variant of the above algorithm uses the preconditioner $P_t' := (A^{\top} D_t A)^{\dagger} A^{\top} \sqrt{D_t}$ instead of $P_t$. 
If $A$ is sparse, rather than computing the preconditioner explicitly, one can execute a run of AGD by solving the corresponding sparse linear system at each iteration. Theorem \ref{th:main} (which applies both for the preconditioner $P_t$ and $P_t'$) then shows that in total this algorithm needs to solve $O^*(n^{|1/2-1/p|})$ sparse linear systems.

\subsection{Initial Point and Termination Conditions}
\ifshortversion
\else
We start by showing that $x(t_0)$ is easy to compute for $t_0$ large enough.
\fi
\begin{lemma} \label{lem:init}
For $t^{p-1}>\frac{2}{p}c^{\top}(A^{\top}A)^{\dagger}c$ and $t>2\norm b_{2}$,
we have
\mideq{x({t})=(A^{T}A)^{\dagger}A^{T}b-\frac{1}{p}t^{2-p}(A^{\top}A)^{\dagger}c.}
\end{lemma}
\fullversion{
\begin{proof}
Let $x=(A^{T}A)^{\dagger}A^{T}b-\frac{1}{p}t^{2-p}(A^{\top}A)^{\dagger} c$.
Note that
\begin{align*}
\norm{Ax-b}_{2} & =\norm{(A(A^{\top}A)^{\dagger}A^{\top}-I)b-\frac{1}{p}t^{2-p}A(A^{\top}A)^{\dagger}c}_{2}\\
 & \leq\norm b_{2}+\frac{1}{p}t^{2-p}c^{\top}(A^{\top}A)^{\dagger}c<t
\end{align*}
where we used the assumption on $t$ at the end. By the definition
of $f_{t}$, we have that
\[
\sum_{i=1}^{n}f_{t}((Ax-b)_{i})=\frac{p}{2}t^{p-2}\norm{Ax-b}_{2}^{2}.
\]
Checking the KKT condition, we can see that $x$ is indeed the minimizer
of $\min_{x}c \cdot x+\sum_{i=1}^{n}f_{t}((Ax-b)_{i})$.
\end{proof}

Next we observe that for $t$ small, $x(t)$ is indeed close to optimal.
}
\begin{lemma} \label{lem:approx}
For any $t \geq 0$ one has
$$c \cdot x(t) + \|A x(t) - b\|_p^p \leq n (\frac{p}{2} - 1) t^p + \min_{x \in \R^d} c \cdot x + \|A x - b\|_p^p ~.$$
\end{lemma}
\fullversion{
\begin{proof}
Since $f_t(s) \geq |s|^p$ for any $s \in \R$, one has the following sequence of inequalities for any $x \in \R^d$:
\begin{eqnarray*}
c \cdot x(t) + \|A x(t) - b\|_p^p & \leq & c \cdot x(t) + \sum_i f_t((A x(t) - b)_i) \\
& \leq & c \cdot x + \sum_i f_t((A x - b)_i) \\
& \leq & c \cdot x + \|A x - b\|_p^p + n \|f_t - |\cdot|^p\|_{\infty} ~.
\end{eqnarray*}
It only remains to verify that $\|f_t - |\cdot|^p\|_{\infty} = (\frac{p}{2} - 1) t^p$.
\end{proof}
}
%

\begin{theorem} \label{th:main}
With the initial parameter $t_{0}=\max((2c^{\top}(A^{\top}A)^{\dagger}c)^{\frac{1}{p-1}},2\|b\|_{2})$,
the algorithm finds a point $x(t_{k})$ such that
\[
c\cdot x(t_{k})+\|Ax(t_{k})-b\|_{p}^{p}\leq\varepsilon+\min_{x\in\R^{d}}c\cdot x+\|Ax-b\|_{p}^{p}
\]
in $k=O(1)\cdot\log(\frac{np}{\varepsilon}t_{0}^{p})$ phases.

Furthermore,
each run of AGD terminates after $O_{p}(n^{|\frac{1}{2}-\frac{1}{p}|}\log(n))$
iterations and each step of AGD involves applying $P_t$ or $P'_t$ constant many times plus some linear time work.
\end{theorem}

\begin{proof}
Lemma \ref{lem:init} shows that $x(t_{0})$ can be computed by a linear system.
Lemma \ref{lem:approx} shows that $x(t_{k})$ satisfies the requirement for
$t_{k}\leq(\frac{\varepsilon}{np})^{\frac{1}{p}}$. Since $t_{k}$
is decreased by $1-\frac{1}{2p}$ factor in each step, this gives
the bound on the number of phases.

For the number of iterations in each phase, Lemma \ref{lem:condg} shows that
the condition number $\kappa$ of the problem is $O_{p}(n^{|1-\frac{2}{p}|})$.
Therefore, AGD decreases the $\ell_{2}$ distance by a constant factor
for every $O_{p}(n^{|1/2-1/p|})$ iterations (Lemma \ref{lem:AGD}). Note that $x(t_{i+1})$
is used only for constructing the quadratic extension. In particular,
we only need to find $x\in\cN_{s(t_{i+1})}(c\cdot\gamma)$ for some
constant $c$. Due to the preconditioning $P_{t}$, we only need find
$y$ that is closer to $y(t_{i+1})$ in $\ell_{\infty}$ norm by some
$O_p(1)$ constant. This can be achieved by decreasing $\ell_{2}$
norm by $O_p(1/n^{O(1)})$. This gives the extra $O_p(\log(n))$ factor.
\end{proof}
\red{We do not give an explicit explanation on how small error we need to take for AGD because the number of iterations depends on $\log(1/\eps)$ and it is easy to see that $\eps = O_p(1/n^{O(1)})$ is enough and it will only affect the final runtime by a logarithmic factor.}

\subsection{Proof of Lemma \ref{lem:cond}} \label{sec:prooflemcond}
To shorten notation we write $H_t := \nabla^2 f_t(s(t))$.
\fullversion{We start with a lemma showing that $x(t)$ satisfies a certain differential
equation.}
\begin{lemma}[Dynamic of the homotopy path]
\label{lem:dynamic_path} One has
\mideq{\frac{dx_{t}}{dt}= - (A^{\top}H_{t}A)^{\dagger}A^{\top}(\frac{d}{dt}f_{t}')(s_{t}) ~.}
\end{lemma}
\begin{proof}
The KKT condition for $x(t)$ is given by
\mideq{c+A^{\top}f_{t}'(s({t}))=0.}
Taking derivatives with respect to $t$ on both sides, we have that
\mideq{A^{\top}H_{t}A\frac{dx({t})}{dt}+A^{\top}(\frac{d}{dt}f_{t}')(s({t}))=0 ~.}
The proof is concluded by noting that $\mathrm{ker}(A^{\top}H_{t}A)=\mathrm{ker}(A)$ and recalling that $x(t) \in \mathrm{ker}(A)^{\perp}$.
\end{proof}
Using the differential equation, we can bound how fast $s(t)$ is
moving.
\begin{lemma}[Speed of the homotopy path]
\label{lem:s_speed} For any $i \in [n]$ one has
\begin{align*}
\left|\frac{ds_{i}(t)}{dt}\right| & \leq \frac{p^2}{p-1} \sqrt{n} \left(t/|s_{i}(t)|\right)^{\frac{p-2}{2}} .
\end{align*}
\end{lemma}
\begin{proof}
Using Lemma \ref{lem:dynamic_path} and that $H_{t}^{1/2}A(A^{\top}H_{t}A)^{\dagger}A^{\top}H_{t}^{1/2}$
is a projection matrix, we have that
\begin{align}
\norm{H_{t}^{1/2}\frac{ds({t})}{dt}}_{2} & =\norm{H_{t}^{1/2}A(A^{\top}H_{t}A)^{\dagger}A^{\top}(\frac{d}{dt}f_{t}')(s({t}))}_{2}\nonumber \\
 & \leq\norm{H_{t}^{-1/2}(\frac{d}{dt}f_{t}')(s({t}))}_{2} ~. \label{eq:Hdx}
\end{align}
To estimate the last term, we use the formula of $f_{t}$ and note
that
\begin{align*}
f_{t}'(s) & =\begin{cases}
pt^{p-2}s & \text{if }|s|\leq t\\
p|s|^{p-2}s & \text{otherwise}
\end{cases},\\
f_{t}''(s) & =\begin{cases}
pt^{p-2} & \text{if }|s|\leq t\\
p(p-1)|s|^{p-2} & \text{otherwise}
\end{cases},\\
\frac{d}{dt}f_{t}'(s) & =\begin{cases}
p(p-2)t^{p-3}s & \text{if }|s|\leq t\\
0 & \text{otherwise}
\end{cases}.
\end{align*}
Therefore, we have that
\[
\left|(f_{t}''(s_{i}(t)))^{-1/2}\frac{d}{dt}f_{t}'(s_{i}(t))\right|\leq\begin{cases}
\sqrt{p}\left|p-2\right|t^{\frac{p-2}{2}} & \text{if }|s_{i}(t)|\leq t\\
0 & \text{otherwise}
\end{cases}.
\]
Putting this into (\ref{eq:Hdx}) gives
\mideq{\norm{H_{t}^{1/2}\frac{ds_{t}}{dt}}_{2}\leq\sqrt{p}\left|p-2\right|t^{\frac{p-2}{2}}\sqrt{n}.}
Hence, we have that
\begin{align*}
\left|\frac{ds_{i}(t)}{dt}\right| & \leq\left|p-2\right|\sqrt{n}\cdot\begin{cases}
1 & \text{if }|s_{i}(t)|\leq t ~,\\
(p-1)^{-\frac{1}{2}}\left(t/|s_{i}(t)|\right)^{\frac{p-2}{2}} & \text{else.}
\end{cases}
\end{align*}
Simplifying and combining both case, we have the result.
\end{proof}

Equipped with the above lemma we can now move to the proof of Lemma \ref{lem:cond}. 

\begin{proof} of Lemma \ref{lem:cond}.
We start by showing that $s((1-h) t) \in \cN_{s(t)}(\gamma)$. First Lemma \ref{lem:s_speed} gives
\begin{align}
\left|\frac{ds_{i}(t)^{p/2}}{dt}\right| & \leq \frac{p^3}{p-1} \sqrt{n} t^{\frac{p-2}{2}} . \label{eq:yetagainaderiv}
\end{align} 
Thus we have:
\begin{align*}
\left|s_{i}(t)^{p/2}-s_{i}((1-h)t)^{p/2}\right| & \leq\int_{(1-h)t}^{t}\left|\frac{ds_{i}(t')^{p/2}}{dt'}\right|dt'\\
 & \leq\frac{p^{3}}{p-1}\sqrt{n}t^{\frac{p-2}{2}}\cdot ht=\frac{p^{3}}{p-1}\sqrt{n}t^{\frac{p}{2}}h
\end{align*}
This shows that $s((1-h) t) \in \cN_{s(t)}(\gamma)$.


Next we need to argue about the range of $f_{(1-h)t}''(s_i)$ for $s' \in \cN_{s(t)}(\gamma)$. First note that
$$\min(1, p-1) p \max(((1-h)t)^{p/2}, |s'_i|^{p/2})^{2 - 4/p} \leq f_{(1-h)t}''(s'_i) \leq \max(1, p-1) p \max(((1-h)t)^{p/2}, |s'_i|^{p/2})^{2 - 4/p} ~.$$
Using $h = \frac1{2p}$, we have 
\mideq{\alpha_i \leq f_{(1-h)t}''(s'_i) \leq \beta_i ~,}
where $\xi = \mathrm{sign}(p-2)$ and
\begin{align*}
\alpha_i & := \frac{p-1}{2} \max(t^{p/2}, |s_i(t)|^{p/2} - \xi \gamma)^{2 - 4/p} ~, \\
\beta_i & := p^2 \max(t^{p/2}, |s_i(t)|^{p/2} + \xi \gamma)^{2 - 4/p} ~.
\end{align*}
Noting that for any $a >0, b \geq 0$ one has $\frac{\max(a,b+\gamma)}{\max(a,b-\gamma)} \leq \frac{a+2\gamma}{a}$ we get
$$\frac{\beta_i}{\alpha_i} \leq \frac{2p^2}{p-1} \left( \frac{t^{p/2} + 2 \gamma}{t^{p/2}}\right)^{|2-4/p|} ~,$$
which concludes the proof of Lemma \ref{lem:cond}. 
\end{proof}
\section{Input sparsity algorithm} \label{sec:inputsparsity}
In this section we replace AGD in our homotopy method by \emph{mini-batch
Katyusha}, \cite{Zey17}. This is the current fastest algorithm for minimizing convex functions of the form $\sum_i f_i(x)$. \red{To just obtain an input-sparsity time algorithm, there are many other options such as \cite{lin2014accelerated,johnson2013accelerating,lin2015universal,frostig2015regularizing}. The key benefit of these algorithms is that its runtime has smaller dependence on $n$ compared to AGD.}
\begin{lemma}[Theorem 5.2 in \cite{Zey17}]
\label{lem:Katyusha_mini} For $i \in [n]$ let $F_{i}$ be a $L_{i}$ smooth convex
function on $\R^{d}$, and let $F=\sum_{i\in[n]}F_{i}$. Suppose
that $F$ is $\sigma$ strongly convex and $L$ smooth. Given an initial
point $x_{0}$, an error parameter $0<\varepsilon<\frac{1}{2}$, and
a batch-size $b$, the \emph{mini-batch Katyusha} algorithm outputs
$x$ such that
\[
\E F(x)-\min_{x}F(x)\leq\varepsilon(F(x_{0})-\min_{x}F(x))
\]
in $O(\frac{n}{b}+\sqrt{\frac{L}{\sigma}}+\frac{1}{b}\sqrt{\frac{n\cdot\sum L_{i}}{\sigma}})\log(\frac{1}{\varepsilon})$
iterations. Each iteration involves computing $\sum_{i\in S}\nabla F_{i}(x)$
where $S$ is a set of $b$ numbers in $[n]$ chosen at random with
replacement.
\end{lemma}
\fullversion{
\begin{remark}
Instead of strongly convex in $\R^{n}$, it suffices to have a
subspace $H$ such that $F$ is $\sigma$-strongly convex on $H^{\perp}$
and that $F_{i}$ is constant on the subspace $H$, namely, $F_{i}(x+y)=F_{i}(x)$
for all $x\in\R^{d}$, $y\in H$ and $i\in[n]$. 
\end{remark}}
In our case, we use $F_{i}(y)=c\cdot P_{t}''y+\tilde{f}_{t,i}(a_{i}\cdot P_{t}''y-b_{i})$
where $P_{t}''$ is a new preconditioner to be defined. We cannot
use $P_{t}$ or $P_{t}'$ because they are too costly for input sparsity
time algorithms.
\begin{lemma}
\label{lem:SVRG_F}Define $P_{t}''=(A^{\top}W_{t}A)^{\dagger}A^{\top}\sqrt{W_{t}}$
where $A^{\top}W_{t}A$ is a spectral sparsifier of $A^{\top}D_{t}A$,
namely, $W_{t}$ is a diagonal matrix with $O(d)$ non-zeros
such that 
\[
\frac{1}{2}A^{\top}D_{t}A\preceq A^{\top}W_{t}A\preceq2A^{\top}D_{t}A.
\]
Let $F_{i}(y)=c\cdot P_{t}''y+\tilde{f}_{t,i}(a_{i}\cdot P_{t}''y-b_{i})$.
Then, each iteration of mini-batch Katyusha on $\sum_{i}F_{i}$ takes
$\tilde{O}(\nnz(A)\frac{b}{n}+d^{2})$ expected time plus a $\tilde{O}(\nnz(A)+d^{\omega})$
preprocessing time
\end{lemma}

\begin{proof}
The diagonal $W_{t}$ can be find in $\tilde{O}(\nnz(A)+d^{\omega})$ time \cite{lee2017sdp}.
\red{To get a slightly denser diagonal, one can use \cite{cohen2015uniform,spielman2011graph,drineas2006subspace} instead.}
Also, we can precompute $(A^{\top}W_{t}A)^{\dagger}$ and
store it as a dense matrix in each phase. This takes $O(d^{\omega})$
time.

To compute
\[
\sum_{i\in S}\nabla F_{i}(x)=|S|P_{t}''^{\top} c+ P_t'' \sum_{i\in S}\tilde{f}_{t,i}'(a_{i}\cdot P_{t}''y-b_{i})a_{i},
\]
we can compute $P_{t}''^{\top} c$ and $P_{t}''y$ first. Since
we have already computed $(A^{\top}W_{t}A)^{\dagger}$, it only takes
$O(d^{2})$ to compute both $P_{t}''^{\top} c$ and $P_{t}''y$. Then, we
can compute the rest in time linear to the total number of non-zeros
in $a_{i}$ for $i\in S$ plus another $O(d^2)$ time multiplication by $P_t''$. Since $S$ is a random set of size $b$,
the total non-zeros is $O(\nnz(A)\frac{b}{n})$ in expectation. Therefore,
it takes in total
\mideq{O(\nnz(A)\frac{b}{n}+d^{2})}
time. 
\end{proof}
Now, it remains to bound the smoothness of $F_{i}$. 
\begin{lemma}
\label{lem:SVRG_smooth}Using the same notation as Lemma \ref{lem:SVRG_F},
we have that $L=O_{p}(n^{|1-2/p|})$, $\sigma=\Omega(1)$ and $\sum_{i\in[n]}L_{i}=O_{p}(n^{|1-2/p|}d)$.
\end{lemma}

\begin{proof}
The bound on $L$ and $\sigma$ follows from Lemma \ref{lem:condg}. For
the bound on $L_{i}$, we note that $\tilde{f}_{t,i}''\leq\kappa D_{t,ii}$
using Lemma \ref{lem:cond}. Therefore, we have 
\begin{align*}
\nabla^{2}F_{i}(y) & \preceq\kappa D_{t,ii}\cdot(P_{t}'')^{\top}a_{i}a_{i}^{T}(P_{t}'')
 \preceq\kappa D_{t,ii}\cdot a_{i}^{\top}P_{t}''(P_{t}'')^{\top}a_{i}\cdot I
\end{align*}
For the last term, using the definition of $P_{t}''$ and the fact
that $A^{\top}W_{t}A$ is a spectral sparsifier of $A^{\top}D_{t}A$,
we have that 
\begin{align*}
a_{i}^{\top}P_{t}''(P_{t}'')^{\top}a_{i} & =a_{i}^{\top}(A^{\top}W_{t}A)^{\dagger}A^{\top}W_{t}A(A^{\top}W_{t}A)^{\dagger}a_{i}
 = a_{i}^{\top}(A^{\top}W_{t}A)^{\dagger}a_{i}
  \leq 2a_{i}^{\top}(A^{\top}D_{t}A)^{\dagger}a_{i}
\end{align*}
Therefore, we have that 
\[
L_{i}\leq 2\kappa(\sqrt{D_{t}}A(A^{\top}D_{t}A)^{\dagger}A^{\top}\sqrt{D_{t}})_{ii}.
\]
Since $\sqrt{D_{t}}A(A^{\top}D_{t}A)^{\dagger}A^{\top}\sqrt{D_{t}}$
is a projection matrix with rank $d$, we have that
\[
\sum_{i}(\sqrt{D_{t}}A(A^{\top}D_{t}A)^{\dagger}A^{\top}\sqrt{D_{t}})_{ii}=d.
\]
This gives the result.
\end{proof}
Now, we can use Lemma \ref{lem:SVRG_F} and Lemma \ref{lem:SVRG_smooth}
in Lemma \ref{lem:Katyusha_mini} and get the following result:
\begin{theorem}
\label{thm:input_sparsity}
We can find $x$ such that
\[
c\cdot x+\norm{Ax-b}_{p}^{p}\leq\min_{x}c\cdot x+\norm{Ax-b}_{p}^{p}+\varepsilon.
\]
in time
\[
\widetilde{O}_{p}\left[(\nnz(A) \left(1+n^{\left|\frac{1}{2}-\frac{1}{p}\right|}\sqrt{\frac{d}{n}}\right)+n^{\left|\frac{1}{2}-\frac{1}{p}\right|}d^{2}+d^{\omega})\log \left(\frac{t_{0}^{p}}{\varepsilon}\right)\right]
\]
where $\ensuremath{t_{0}=\max((2c^{\top}(A^{\top}A)^{\dagger}c)^{\frac{1}{p-1}},2\|b\|_{2})}.$
Writing it in the input sparsity form, we have
\[
\widetilde{O}_{p}\left[(\nnz(A)+d^{\frac{1}{2}\max(p,\frac{p}{p-1})+1}+d^{\omega})\log \left(\frac{t_{0}^{p}}{\varepsilon}\right)\right].
\]
\end{theorem}

\ifshortversion
\else
\begin{proof}
As we argued in Theorem \ref{th:main}, it suffices to solve it up
to $O_{p}(1/n^{O(1)})$ accuracy in each phase. Using Lemma \ref{lem:SVRG_smooth}
into Lemma \ref{lem:Katyusha_mini}, we have that Katyusha takes 
\begin{align*}
 & \widetilde{O}_{p}\left(\frac{n}{b}+\sqrt{\kappa}+\frac{1}{b}\sqrt{n\kappa d}\right)
\end{align*}
iterations. Now using Lemma \ref{lem:SVRG_F}, we know that the total
time is 
\begin{align*}
& \widetilde{O}_{p}\left[\left(\frac{n}{b}+\sqrt{\kappa}+\frac{1}{b}\sqrt{n\kappa d}\right) \left(Z\frac{b}{n}+d^{2}\right)+d^{\omega}\right] \\
& =  
\widetilde{O}_{p}\left[Z\left((1+ \sqrt{\frac{\kappa d}{n}}\right) + d^{\omega} + d^2 \sqrt{\kappa} + \frac{d^2 \sqrt{n}}{b} \sqrt{\kappa d + n} + Z \sqrt{\kappa} \frac{b}{n}  \right] ~.
\end{align*}
where $Z=\nnz(A)$ is the total number of non-zeros in $A$. We now choose $b$ to optimize the term
\begin{equation} \label{eq:lastminute}
 \frac{d^2 \sqrt{n}}{b} \sqrt{\kappa d + n} + Z \sqrt{\kappa} \frac{b}{n} ~.
\end{equation}
If $\kappa d\geq n$, then we choose $b=\left\lceil \sqrt{\frac{n^{\frac{3}{2}}d^{\frac{5}{2}}}{Z}}\right\rceil $
and we find that \eqref{eq:lastminute} is equal to (up to factor $2$) $\sqrt{Z} d^{5/4} n^{-1/4}$ which is always smaller than $Z \sqrt{\frac{\kappa d}{n}} + d^2$.
If $\kappa d\leq n$, then we choose $b=\left\lceil \sqrt{\frac{n^{2}d^{2}}{Z\sqrt{\kappa}}}\right\rceil $
and we find that \eqref{eq:lastminute} is equal to (up to factor $2$) $\sqrt{Z} d \kappa^{-1/4}$ which is also always smaller than $Z \sqrt{\frac{\kappa d}{n}} + d^2$.
Combining both cases, we have that the cost per phase is 
\[
\widetilde{O}_{p}\left[Z\left(1+\sqrt{\frac{\kappa d}{n}}\right)+\sqrt{\kappa}d^{2}+d^{\omega}\right].
\]
By Theorem \ref{th:main}, we know that the number of phase is $O(1)\cdot\log(\frac{np}{\varepsilon}t_{0}^{p})$.
This gives the first result.

To write the running time in input sparsity, we note that $\kappa d\geq n$
implies
\[
n\leq\begin{cases}
O_{p}(d^{\frac{p}{2}}) & \text{if }p\geq2\\
O_{p}(d^{\frac{1}{2-\frac{2}{p}}}) & \text{if }p\leq2
\end{cases}.
\]
For $p\geq2$, we have that
\[
\widetilde{O}_{p}\left[Z(1+\sqrt{\frac{\kappa d}{n}})\right]=\widetilde{O}_{p}\left[Z+nd\sqrt{\frac{\kappa d}{n}}\right]=\widetilde{O}_{p}\left[Z+d^{\frac{p}{2}+1}\right].
\]
Also, we note that
\[
\sqrt{\kappa}d^{2}\leq O_{p}(n^{\frac{1}{2}-\frac{1}{p}}d^{2})\leq O_{p}(n+(d^{2})^{\frac{1}{\frac{1}{2}+\frac{1}{p}}})=O_{p}(n+d^{\frac{4}{1+\frac{2}{p}}}).
\]
Combining both terms, we have
\begin{align*}
\widetilde{O}_{p}\left[Z(1+\sqrt{\frac{\kappa d}{n}})+\sqrt{\kappa}d^{2}+d^{\omega}\right] & =\widetilde{O}_{p}\left[Z+d^{\frac{p}{2}+1}+d^{\frac{4}{1+\frac{2}{p}}}+d^{\omega}\right]
  =\widetilde{O}_{p}\left[Z+d^{\frac{p}{2}+1}+d^{\omega}\right].
\end{align*}

Similarly, for $p\leq2$, the total running time is
\begin{align*}
\widetilde{O}_{p}\left[Z(1+\sqrt{\frac{\kappa d}{n}})+\sqrt{\kappa}d^{2}+d^{\omega}\right] & =\widetilde{O}_{p}\left[Z+d^{\frac{p}{2(p-1)}+1}+d^{\omega}\right].
\end{align*}
\end{proof}
\fi



\section{Self-concordance lower bound for $\ell_p^n$ balls} \label{sec:lb}
We first recall the definition, introduced in \cite{NN94}, of a self-concordant barrier. 

\begin{definition}
A function $\Phi: \mathrm{int}(\cK) \rightarrow \R$ is a barrier for $\cK$ if
$$\Phi(x) \xrightarrow[x \to \partial \cK]{} +\infty ~.$$
A $C^3$-smooth convex function $\Phi : \mathrm{int}(\cK) \rightarrow \R$ is self-concordant if for all $x \in  \mathrm{int}(\cK), h \in \R^n$,
\begin{equation} \label{eq:sc}
\nabla^3 \Phi(x) [h,h,h] \leq 2 (\nabla^2 \Phi(x)[h,h])^{3/2} ~.
\end{equation}
Furthermore it is $\nu$-self-concordant if in addition for all $x \in  \mathrm{int}(\cK), h \in \R^n$,
\begin{equation} \label{eq:nusc}
\nabla \Phi(x)[h] \leq \sqrt{\nu \cdot \nabla^2 \Phi(x)[h,h] } ~.
\end{equation}
\end{definition}

We also recall that for any convex body in $\R^n$ there exists a self-concordant barrier with self-concordance parameter $\nu =O(n)$. Furthermore, for $\ell_p^n$ balls, \cite{NN94,XY99} and Section 2.g in \cite{alizadeh2003second} showed that there even exists a {\em computationally efficient} barrier with such self-concordance parameter.
Our main theorem in this section is to show that the latter result is essentially unimprovable:
\begin{theorem} \label{th:lbscb}
Let $\Phi$ be a $\nu$-self-concordant barrier on the unit ball of $\ell_p^n$, $p >2$. Assume that $\Phi$ is symmetric in the sense that
$$\Phi(x_1,\hdots, x_n) = \Phi(|x_1|,\hdots,|x_n|) ~.$$
Then one has $\nu \geq \frac{n}{(O(p) \log(n))^{\frac{p}{p-2}}}$.
\end{theorem}
\begin{remark}
Let $q$ be the conjugate of $p$. Since we can construct a $O(\nu)$-self-concordant barrier function for $\ell_{q}^{n}$ using the barrier for $\ell_p$ (see Thm 2.4.4 and Prop 5.1.4 in \cite{NN94}), we also have an almost linear lower bound for the case $p<2$.
\end{remark}
We conjecture that the result holds without the symmetry assumption on $\Phi$. In fact there may even be a deeper reason why the ``optimal" self-concordant barrier for a ``symmetric" body should be ``symmetric", but we are not aware of any existing such result. At the moment without the symmetry assumption we can prove a $\tilde{\Omega}(n^{1/3})$ lower bound.

Let us now recall some general properties of self-concordant barriers.

\begin{theorem}[Prop 2.3.2 in \cite{NN94}, Sec 2.2 in \cite{Nem04b}] \label{th:scb}
Let $\Phi$ be $\nu$-self-concordant barrier for $\cK$. The following holds true.
\begin{enumerate}
\item For any $x, y \in \mathrm{int}(\cK)$,
\mideq{\Phi(y) - \Phi(x) \leq \nu \log \left( \frac{1}{1-\pi_x(y)} \right) ~,}
where $\pi_x(y)$ is the Minkowski gauge, i.e., $\pi_x(y) = \inf\{t>0 : x + \frac{1}{t} (y-x) \in \cK\}$.
\item For any $x \in \mathrm{int}(\cK)$ and $h$ such that $\|h\|_x \leq 1/2$,
\mideq{D_{\Phi}(x+h, x) \leq \|h\|_x^2 ~.}
\item For any $x \in \mathrm{int}(\cK)$ and $h$ such that $\|h\|_x \leq 1/2$,
\mideq{D_{\Phi}(x+h, x) \geq \frac{1}{4} \|h\|_x^2 ~.}
\item For $x \in \mathrm{int}(\cK)$ and $r>0$ let $W_r(x) = \{x + h : \|h\|_x < r\}$ be the Dikin ellipse of radius $r$ at $x$. Then one always has $W_1(x) \subset \cK$.
\end{enumerate}
\end{theorem}

Before moving to the proof of Theorem \ref{th:lbscb} we make a few observations. An important property of self-concordant barriers not listed above is that they give a $O(\nu)$ rounding of the body in the sense that at the center $x^* = \argmin_x \Phi(x)$ one has $W_1(x^*) \subset \cK \subset W_{O(\nu)} (x^*)$ (See Theorem 4.2.6 in \cite{Nes04}). This directly implies that for the unit ball of $\ell_p^n$ one must have $\nu = \Omega(n^{1/2 - 1/p})$. In fact the following simple random walk argument improves this trivial bound to $\nu = \Omega(n^{1-2/p})$. First note that
thanks to Theorem \hyperref[th:scb]{\ref*{th:scb}.1} it suffices to find $x$ with say $\|x\|_p = 1/2$ and with $\Phi(x) - \Phi(0) = \Omega(n^{1-2/p})$. Next let $X_0=0$, and $X_i= X_{i-1}+\frac{1}{2 n^{1/p}} \xi_i e_i$ with $(\xi_i)_{i \in [n]}$ i.i.d. Rademacher random variables (notice that $\|X_n\|_p = 1/2$). Then (using Theorem \hyperref[th:scb]{\ref*{th:scb}.3} for the inequality):
\begin{eqnarray*}
\E \Phi(X_n) - \Phi(0) & = & \E \sum_{i=1}^n (\Phi(X_i) - \Phi(X_{i-1})) = \E \sum_{i=1}^n D_{\Phi}(X_i, X_{i-1}) \\ 
& \geq &  \E \sum_{i=1}^n \min\{\frac{1}{4} \|X_i - X_{i-1}\|_{X_{i-1}}^2,\frac{1}{16}\} \\
& = & \E \sum_{i=1}^n \min\{\frac{1}{16 n^{2/p}}\|e_i\|_{X_{i-1}}^2,\frac{1}{16}\} ~.
\end{eqnarray*}
Crucially we now observe that Theorem \hyperref[th:scb]{\ref*{th:scb}.4} shows that $\|e_i\|_{X_{i-1}} \geq 1$ (since $X_{i-1} + e_i$ is outside of the $\ell_p^n$ unit ball), which concludes the proof of the following lemma.
\begin{lemma} \label{lem:lbscb}
Let $\Phi$ be a $\nu$-self-concordant barrier on the unit ball of $\ell_p^n$, $p >2$. One has $\nu = \Omega(n^{1-2/p})$.
\end{lemma}

From a high level point of view, one can hope to improve the above argument using that the Hessian at $X_{i-1}$ should intuitively increase with $i$, meaning that $\|e_i\|_{X_{i-1}}$ could be potentially much bigger than $1$. We formalize this idea in the following proof of Theorem \ref{th:lbscb}.

\begin{proof} of Theorem \ref{th:lbscb}.
Fix $\epsilon = \frac{1}{2 n^{1/p}}$ and let $(X_i)_{i \in \{0,\hdots,n\}}$ be defined as above (observe that by the symmetry assumption $\Phi(X_i)$ is in fact a non-random quantity) and recall that we proved
\begin{equation} \label{eq:scbproof1}
\Phi(X_i) - \Phi(X_0) \geq \Phi(X_{i-1}) - \Phi(X_0) + \min\{\frac{\epsilon^2}{4} \|e_i\|_{X_{i-1}}^2,\frac{1}{16}\} ~.
\end{equation}
We denote by $\nu_{i,n}$ the infimum of $\Phi(X_i) - \Phi(X_0)$ over all self-concordant barriers $\Phi$ on the unit ball of $\ell_p^n$. Note that $\nu_{i,n}$ is increasing with respect to both indices $i$ and $n$. We now consider two cases, depending on whether $\|e_i\|_{X_{i-1}}$ is larger than $c/\epsilon$ or not, where $c \in (0,1)$ will be fixed later. If it is larger then we will simply use \eqref{eq:scbproof1}, so let us assume that it is smaller. Then we know that (using Theorem \hyperref[th:scb]{\ref*{th:scb}.3})
$$\Phi\left(X_{i-1} + \frac{\epsilon}{2 c} e_i \right) - \Phi(X_{i-1}) \leq \frac{\epsilon^2}{4 c^2} \|e_i\|_{X_{i-1}}^2 ~,$$
and in particular multiplying this equation by $c^2$ and using \eqref{eq:scbproof1} (as well as $\Phi(X_{i-1}) \leq \Phi(X_i)$) we get
\begin{eqnarray*}
(1+c^2) \Phi(X_i) & \geq & \Phi(X_{i-1}) + \frac{\epsilon^2}{4} \|e_i\|_{X_{i-1}}^2 + c^2 \Phi(X_{i-1}) \\
& \geq & \Phi(X_{i-1}) + c^2 \Phi\left(X_{i-1} + \frac{\epsilon}{2 c} e_i \right) ~.
\end{eqnarray*}
Let us now consider the function $\psi : e_i^{\perp} \rightarrow \overline{\R}$ defined by $\psi(z) = \Phi\left(z+\frac{\epsilon}{2 c} e_i \right)$. Clearly $\psi$ is a symmetric self-concordant barrier on an $\ell_p^{n-1}$ ball of radius $R:=(1-(\epsilon/2c)^p)^{1/p}$, and thus by convexity of $\psi$ and the definition of $\nu_{i,n}$,
$$R (\psi(X_{i-1}) - \psi(0)) \geq \psi(R X_{i-1}) - \psi(0) \geq \nu_{i-1,n-1} ~.$$
Finally putting the above together with $\psi(0) \geq \Phi(X_0)$ and $1 / R \geq \left(1+\left(\frac{\epsilon}{2c}\right)^p \right)^{1/p} \geq 1+\frac{1}{2p}\left(\frac{\epsilon}{2c}\right)^p$ we proved that either
$$\Phi(X_i) - \Phi(X_0) \geq \Phi(X_{i-1}) - \Phi(X_0) + \frac{c^2}{4} ~.$$
or 
$$\Phi(X_i) - \Phi(X_0) \geq \frac{1}{1+c^2} (\Phi(X_{i-1}) - \Phi(X_0)) + \frac{c^2}{1+c^2} \left(1+\frac{1}{2p}\left(\frac{\epsilon}{2c}\right)^p \right) \nu_{i-1,n-1} ~.$$
In particular we showed that, for $i=n$, with $\nu_n := \nu_{n,n}$, 
$$\nu_{n}\geq\nu_{n-1}+\min\left(\frac{c^{2}}{4},\frac{c^{2}}{1+c^{2}}\frac{1}{2p}\left(\frac{\epsilon}{2c}\right)^{p}\nu_{n-1}\right)~,$$
which implies in particular (by simply checking when the minimum in the above equation is attained at the first term)
$$\nu_{n}\geq\min\left(\frac{(1+c^{2})p}{2}\left(\frac{2c}{\epsilon}\right)^{p},\left(1+\frac{c^{2}}{1+c^{2}}\frac{1}{2p}\left(\frac{\epsilon}{2c}\right)^{p}\right)\nu_{n-1}\right)~,$$
which means by induction (since $\nu_1 \geq \epsilon^2 /4$ by \eqref{eq:scbproof1} and $\|e_1\|_{X_0} \geq 1$)
$$\nu_{n}\geq\min\left(\frac{(1+c^{2})p}{2}\left(\frac{2c}{\epsilon}\right)^{p},\left(1+\frac{c^{2}}{1+c^{2}}\frac{1}{2p}\left(\frac{\epsilon}{2c}\right)^{p}\right)^{n-1}\frac{\varepsilon^{2}}{4}\right)~.$$
Taking $c = \Theta\left((\Theta(p) \log(n))^{\frac{-1}{p-2}}\right)$ concludes the proof with trivial calculations.
\end{proof}

\begin{theorem} \label{th:lbscb2}
Let $\Phi$ be a $\nu$-self-concordant barrier on the unit ball of $\ell_p^n$, $p >2$. Then one has $\nu \geq \frac{n^{1/p}}{(O(p) \log(n))^{\frac{p}{p-2}}}$.
Combining with Lemma \ref{lem:lbscb}, we have that $\nu = \Omega \left( \frac{n^{1/3}}{\log(n)}\right)$.
\end{theorem}

\ifshortversion
\else
\begin{proof}
The proof for the asymmetric case is essentially the same as the symmetric
case. Again, let $\varepsilon=\frac{1}{2n^{1/p}}$ and let $X_{i}$
be the random process defined previously. However, we note that $\Phi(X_{i})$
is now a random variable, unlike in the previous proof.

The main difference is that without the symmetry assumption, $\|e_{i}\|_{X_{i-1}}$
depends not only on $i$, but also on $X_{i-1}$. Hence, we separate
the cases to $\|e_{i}\|_{X_{i-1}}$ is large for some $X_{i-1}$ and
$\|e_{i}\|_{X_{i-1}}$ is small for all $X_{i-1}$.

For the first case ($\|e_{i}\|_{X_{i-1}}\leq\frac{c}{\varepsilon}$
for some $X_{i-1}$ and some $i$), we do the same calculation
as \eqref{eq:scbproof1} and obtain
\[
\E\left[\Phi(X_{i})|X_{i-1}\right]-\Phi(X_{i-1})\geq\frac{c^{2}}{4}.
\]
Instead of summing the difference for each step, we simply note
that $\pi_{X_{i-1}}(X_{i})=\Omega(n^{-1/p})$. This gives that $\nu=\Omega(n^{1/p}c^{2})$.

For the second case ($\|e_{i}\|_{X_{i-1}}\leq\frac{c}{\varepsilon}$
for all possible $X_{i-1}$ and $i \in [n]$), the
previous proof still holds and we get that 
\[
\E\left[\Phi(X_{i})|X_{i-1}\right]-\Phi(X_{0})\geq\frac{1}{1+c^{2}}(\Phi(X_{i-1})-\Phi(X_{0}))+\frac{c^{2}}{1+c^{2}}\left(1+\frac{1}{2p}\left(\frac{\varepsilon}{2c}\right)^{p}\right)v_{i-1,n-1}
\]
where $v_{i-1,n-1}$ is now the infimum of $\E\Phi(X_{i})-\Phi(X_{0})$
over possibly asymmetric self concordant barriers $\Phi$ on $\ell_{p}^{n}$
and the expectation is taken over all possible $\pm\varepsilon$ for
the first $i$ coordinate and zero otherwise. Since the Hessian bound
holds for all $X_{i-1}$, we can repeat the argument and get that
\begin{align*}
v_{n,n} & \geq\left(1+\frac{c^{2}}{1+c^{2}}\frac{1}{2p}\left(\frac{\varepsilon}{2c}\right)^{p}\right)v_{n-1,n-1}\\
 & \geq\left(1+\frac{c^{2}}{1+c^{2}}\frac{1}{2p}\left(\frac{\varepsilon}{2c}\right)^{p}\right)^{n-1}v_{1,1}\\
 & \geq\left(1+\frac{c^{2}}{1+c^{2}}\frac{1}{2p}\left(\frac{\varepsilon}{2c}\right)^{p}\right)^{n-1}\frac{\varepsilon^{2}}{4}.
\end{align*}
Setting $c=O((\Theta(p)\log n)^{\frac{-1}{p-2}})$, we have that $v_{n,n}\geq n$
and that implies that $\nu=\Omega(n)$.

Combining both cases, we have that $\nu=\Omega(n^{1/p}(\Theta(p)\log n)^{\frac{-1}{p-2}})$.
\end{proof}
\fi
\subsection*{Acknowledgement}
We thank Farid Alizadeh, Arkadi Nemirovski, Aaron Sidford, Yinyu Ye and Yuriy Zinchenko for helpful discussions. The third author have worked/discussed with some of them on constructing an self-concordance barrier for $\ell_p$ ball. This work was supported in part by NSF award CCF-1740551 and CCF-1749609.
\bibliographystyle{plainnat}
\bibliography{newbib}

\begin{thebibliography}{22}
\providecommand{\natexlab}[1]{#1}
\providecommand{\url}[1]{\texttt{#1}}
\expandafter\ifx\csname urlstyle\endcsname\relax
  \providecommand{\doi}[1]{doi: #1}\else
  \providecommand{\doi}{doi: \begingroup \urlstyle{rm}\Url}\fi

\bibitem[Alizadeh and Goldfarb(2003)]{alizadeh2003second}
Farid Alizadeh and Donald Goldfarb.
\newblock Second-order cone programming.
\newblock \emph{Mathematical programming}, 95\penalty0 (1):\penalty0 3--51,
  2003.

\bibitem[{Allen-Zhu}(2017)]{Zey17}
Z.~{Allen-Zhu}.
\newblock {Katyusha: The First Direct Acceleration of Stochastic Gradient
  Methods}.
\newblock In \emph{STOC}, 2017.

\bibitem[Clarkson and Woodruff(2013)]{clarkson2013low}
Kenneth~L Clarkson and David~P Woodruff.
\newblock Low rank approximation and regression in input sparsity time.
\newblock In \emph{Proceedings of the forty-fifth annual ACM symposium on
  Theory of computing}, pages 81--90. ACM, 2013.

\bibitem[Clarkson et~al.(2016)Clarkson, Drineas, Magdon-Ismail, Mahoney, Meng,
  and Woodruff]{clarkson2016fast}
Kenneth~L Clarkson, Petros Drineas, Malik Magdon-Ismail, Michael~W Mahoney,
  Xiangrui Meng, and David~P Woodruff.
\newblock The fast cauchy transform and faster robust linear regression.
\newblock \emph{SIAM Journal on Computing}, 45\penalty0 (3):\penalty0 763--810,
  2016.

\bibitem[Cohen and Peng(2015)]{cohen2015p}
Michael~B Cohen and Richard Peng.
\newblock L p row sampling by lewis weights.
\newblock In \emph{Proceedings of the forty-seventh annual ACM symposium on
  Theory of computing}, pages 183--192. ACM, 2015.

\bibitem[Cohen et~al.(2015)Cohen, Lee, Musco, Musco, Peng, and
  Sidford]{cohen2015uniform}
Michael~B Cohen, Yin~Tat Lee, Cameron Musco, Christopher Musco, Richard Peng,
  and Aaron Sidford.
\newblock Uniform sampling for matrix approximation.
\newblock In \emph{Proceedings of the 2015 Conference on Innovations in
  Theoretical Computer Science}, pages 181--190. ACM, 2015.

\bibitem[Dasgupta et~al.(2009)Dasgupta, Drineas, Harb, Kumar, and
  Mahoney]{dasgupta2009sampling}
Anirban Dasgupta, Petros Drineas, Boulos Harb, Ravi Kumar, and Michael~W
  Mahoney.
\newblock Sampling algorithms and coresets for $\backslash$ell\_p regression.
\newblock \emph{SIAM Journal on Computing}, 38\penalty0 (5):\penalty0
  2060--2078, 2009.

\bibitem[Drineas et~al.(2006)Drineas, Mahoney, and
  Muthukrishnan]{drineas2006subspace}
Petros Drineas, Michael~W Mahoney, and S~Muthukrishnan.
\newblock Subspace sampling and relative-error matrix approximation:
  Column-row-based methods.
\newblock In \emph{European Symposium on Algorithms}, pages 304--314. Springer,
  2006.

\bibitem[Frostig et~al.(2015)Frostig, Ge, Kakade, and
  Sidford]{frostig2015regularizing}
Roy Frostig, Rong Ge, Sham Kakade, and Aaron Sidford.
\newblock Un-regularizing: approximate proximal point and faster stochastic
  algorithms for empirical risk minimization.
\newblock In \emph{International Conference on Machine Learning}, pages
  2540--2548, 2015.

\bibitem[Johnson and Zhang(2013)]{johnson2013accelerating}
Rie Johnson and Tong Zhang.
\newblock Accelerating stochastic gradient descent using predictive variance
  reduction.
\newblock In \emph{Advances in neural information processing systems}, pages
  315--323, 2013.

\bibitem[Lee and Sidford(2015)]{lee2015efficient}
Yin~Tat Lee and Aaron Sidford.
\newblock Efficient inverse maintenance and faster algorithms for linear
  programming.
\newblock In \emph{Foundations of Computer Science (FOCS), 2015 IEEE 56th
  Annual Symposium on}, pages 230--249. IEEE, 2015.

\bibitem[Lee and Sun(2017)]{lee2017sdp}
Yin~Tat Lee and He~Sun.
\newblock An sdp-based algorithm for linear-sized spectral sparsification.
\newblock \emph{arXiv preprint arXiv:1702.08415}, 2017.

\bibitem[Lin et~al.(2015)Lin, Mairal, and Harchaoui]{lin2015universal}
Hongzhou Lin, Julien Mairal, and Zaid Harchaoui.
\newblock A universal catalyst for first-order optimization.
\newblock In \emph{Advances in Neural Information Processing Systems}, pages
  3384--3392, 2015.

\bibitem[Lin et~al.(2014)Lin, Lu, and Xiao]{lin2014accelerated}
Qihang Lin, Zhaosong Lu, and Lin Xiao.
\newblock An accelerated proximal coordinate gradient method.
\newblock In \emph{Advances in Neural Information Processing Systems}, pages
  3059--3067, 2014.

\bibitem[Meng and Mahoney(2013)]{meng2013low}
Xiangrui Meng and Michael~W Mahoney.
\newblock Low-distortion subspace embeddings in input-sparsity time and
  applications to robust linear regression.
\newblock In \emph{Proceedings of the forty-fifth annual ACM symposium on
  Theory of computing}, pages 91--100. ACM, 2013.

\bibitem[Nemirovski(2004)]{Nem04b}
A.~Nemirovski.
\newblock Interior point polynomial time methods in convex programming.
\newblock \emph{Lecture Notes}, 2004.

\bibitem[Nesterov(2004)]{Nes04}
Y.~Nesterov.
\newblock \emph{Introductory lectures on convex optimization: A basic course}.
\newblock Kluwer Academic Publishers, 2004.

\bibitem[Nesterov and Nemirovski(1994)]{NN94}
Y.~Nesterov and A.~Nemirovski.
\newblock \emph{Interior-point polynomial algorithms in convex programming}.
\newblock SIAM, 1994.

\bibitem[Sohler and Woodruff(2011)]{sohler2011subspace}
Christian Sohler and David~P Woodruff.
\newblock Subspace embeddings for the l 1-norm with applications.
\newblock In \emph{Proceedings of the forty-third annual ACM symposium on
  Theory of computing}, pages 755--764. ACM, 2011.

\bibitem[Spielman and Srivastava(2011)]{spielman2011graph}
Daniel~A Spielman and Nikhil Srivastava.
\newblock Graph sparsification by effective resistances.
\newblock \emph{SIAM Journal on Computing}, 40\penalty0 (6):\penalty0
  1913--1926, 2011.

\bibitem[Woodruff and Zhang(2013)]{woodruff2013subspace}
David Woodruff and Qin Zhang.
\newblock Subspace embeddings and$\backslash$ell\_p-regression using
  exponential random variables.
\newblock In \emph{Conference on Learning Theory}, pages 546--567, 2013.

\bibitem[Xue and Ye(1999)]{XY99}
G.~Xue and Y.~Ye.
\newblock An efficient algorithm for minimizing a sum of p-norms.
\newblock \emph{SIAM J. on Optimization}, 10\penalty0 (2):\penalty0 551--579,
  1999.

\end{thebibliography}
\end{document}